\documentclass[11pt]{amsart}
\usepackage{amsmath}
\usepackage{amssymb}
\usepackage{epsfig}
\usepackage{graphicx}

\numberwithin{equation}{section}

\input xy
\xyoption{all}

\calclayout
\allowdisplaybreaks[3]

\theoremstyle{plain}
\newtheorem{prop}{Proposition}

\newtheorem{theo}[prop]{Theorem}
\newtheorem{coro}[prop]{Corollary}

\newtheorem{lemm}[prop]{Lemma}

\theoremstyle{definition}
\newtheorem{defi}[prop]{Definition}

\newtheorem{conj}[prop]{Conjecture}

\newtheorem{rema}[prop]{Remark}

\newtheorem{exam}[prop]{Example}

\def\ra{\rightarrow}

\def\cL{{\mathcal L}}

\def\cO{{\mathcal O}}

\def\rk{{\mathrm{rk}}}

\def\bC{{\mathbb C}}

\def\bG{{\mathbb G}}

\def\bP{{\mathbb P}}
\def\bQ{{\mathbb Q}}
\def\bR{{\mathbb R}}
\def\bZ{{\mathbb Z}}
\def\rH{{\mathrm H}}

\def\sB{{\mathsf B}}
\def\sH{{\mathsf H}}

\def\Bl{\mathrm{Bl}}
\def\eff{\mathrm{eff}}
\def\Pic{\mathrm{Pic}}

\def\NS{\mathrm{NS}}

\makeatother
\makeatletter

\author{Brendan Hassett}
\address{Department of Mathematics\\
Rice University, MS 136 \\
Houston, Texas  77251-1892 \\
USA}
\email{hassett@rice.edu}

\author{Sho Tanimoto }
\address{Department of Mathematics\\
Rice University, MS 136 \\
Houston, Texas  77251-1892 \\
USA}
\email{tanimoto@math.rice.edu}

\author{Yuri Tschinkel}
\address{Courant Institute\\
                New York University \\
                New York, NY 10012 \\
                USA }
\email{tschinkel@cims.nyu.edu}
\address{Simons Foundation\\ 160 Fifth Av. \\ New York, NY 10010}

\title[Balanced line bundles]{Balanced line bundles and equivariant compactifications of homogeneous spaces}

\begin{document}
\date{\today}

\begin{abstract}
Manin's conjecture predicts an asymptotic formula for the number 
of rational points of bounded height on a smooth projective variety 
$X$ in terms of global geometric invariants of $X$.
The strongest form of the conjecture implies certain 
inequalities among geometric invariants of $X$ and of its subvarieties. 
We provide a general geometric framework explaining these phenomena, via  
the notion of balanced line bundles, and prove the required inequalities for a large class of 
equivariant compactifications of homogeneous spaces. 
\end{abstract}

\maketitle

\section{Introduction} 
\label{secct:intro}

Let $X$ be a smooth projective variety over a number field. 
It is generally hoped that global geometric properties of $X$
should be reflected in its arithmetic properties. For instance, 
assume that its anticanonical class 
$-K_X$ is ample. It has been conjectured that such $X$ satisfy: 
\begin{quote}
Potential Density: there exists a finite extension  
$F$ of the ground field such that $X(F)$ is Zariski dense (see 
\cite{Harris-T}, \cite{bog-t} for first results in this direction 
and \cite{campana}, \cite{abramovich} for a description of a general framework).
\end{quote}
Supposing that $X$ has dense rational points over $F$, we can ask for quantitative versions
of density:
\begin{quote}
Asymptotic Formulas: Let $\mathcal L=(L,\|\cdot \|)$ be an ample, adelically metrized, line bundle on $X$
and $\sH_{\mathcal L}$ the associated height (for definitions and more background see, e.g., \cite[Section 4.8]{T}). 
Then there exists a Zariski open $X^\circ\subset X$ such that 
\begin{equation*}
\label{eqn:manin}
\# \{ x\in X^\circ(F) \, \mid  \, \sH_{\mathcal L}(x)\le \sB\} \sim c(X,\mathcal L) 
\sB^{a(X,L)} \log(\sB)^{b(X,L)-1}, 
\end{equation*}
as $\sB\ra \infty$. Here $a(X,L)$ and $b(X,L)$ are certain {\em geometric} constants introduced in this context in 
\cite{FMT} and \cite{BM} (and recalled in Section~\ref{sect:general}) 
and $c(X,\mathcal L)$ is a Tamagawa-type number defined in \cite{peyre}, \cite{BT}.  
\end{quote}
When $L=-K_X$ the main term of the asymptotic formula reads
$$
 \# \{ x\in X^\circ(F) \, \mid  \, \sH_{-{\mathcal K}_X}(x)\le \sB\} \sim c(X,-{\mathcal K}_X ) 
\sB \log(\sB)^{\rk\,\Pic(X)-1}
$$
as $\sB \ra \infty$, where $-{\mathcal K}_X$ is the metrized anticanonical bundle.  
For a survey addressing both aspects and containing extensive references, see \cite{T}. 

The Asymptotic Formulas raise many formal questions.  How do we choose $X^{\circ}\subset X$?  
Clearly, we want to exclude subvarieties $Y\subsetneq X$ contributing excessively to the
number of rational points.  For example, if $X$ is a split cubic surface and $L=-K_X$
then lines on $X$ contribute on the order of $\sB^2$ points of height $\le \sB$,
more than the $\sB \log(\sB)^6$ points expected from $X^{\circ}$.  

Furthermore, we
should consider carefully whether to include subvarieties $Y\subsetneq X$ contributing rational points
at the same rate as those from $X^{\circ}$.  For example, if $X\subset \bP^5$ is a complete 
intersection of two quadrics then each line of $X$ contributes on the order of $\sB^2$ points,
the same as the conjectured total for $X^{\circ}$ (see Example~\ref{exam:twobytwo}).
These lines are parametrized by an abelian surface.  Including such subvarieties
must have implications for the interpretation of the Tamagawa-type constant.  

Returning to the case of general $L$, in order for the Asymptotic Formula to be internally
consistent, 
all $Y\subsetneq X$ meeting $X^{\circ}$ must satisfy
$$(a(Y,L|_Y),b(Y,L|_Y) \leq (a(X,L),b(X,L)$$
in the lexicographic order.  Moreover, if the constant $c(X,\mathcal L)$ is to be
independent of the open set $X^{\circ}\subset X$ we must have
\begin{equation}
\label{eqn:yy}
(a(Y, L|_Y)), b(Y, L|_Y)) < (a(X,L), b(X,L)).
\end{equation}
However, there exist varieties of dimension $\ge 3$ 
where these properties fail; these provide 
counterexamples to the Asymptotic Formulas \cite{BT-cubic}. 
On the other hand, no counterexamples are known in the equivariant context, when 
$X$ is an equivariant compactification of a linear algebraic group $G$ or of
a homogeneous space $H\backslash G$, and asymptotic formulas for the number of points
of bounded height have been established for many classes of such compactifications (see \cite{T}).

These arithmetic considerations motivate us to introduce and study the notion of {\em balanced line bundles}
(see Section~\ref{sect:balanced}). In this paper, we establish basic properties of 
balanced line bundles and investigate varieties that carry such line bundles. 
One of our main results is:

\begin{theo}
\label{thm:main}
Let 
$$
H\subset M\subset G
$$ 
be a connected linear algebraic group.
Let $X$ be a smooth projective $G$-equivariant compactification of $H\backslash G$
with big anticanonical line bundle $-K_X$,
and $Y\subset X$ the induced compactification of $H\backslash M$. 
Assume that the projection $G \ra M \backslash G$ admits a rational section.
Then $-K_X$ is balanced with respect to $Y$,
i.e., inequality~\eqref{eqn:yy} holds for $L=-K_X$.
\end{theo}

A version of this geometric result, for $G=\mathbb G_a^n$, appeared in 
\cite[Section 7]{chambert-t02}, where it was used to bound contributions 
from nontrivial characters to the Fourier expansion of height zeta functions 
and, ultimately, to prove asymptotic formulas for the number of rational points of
bounded height (Manin's conjecture) for equivariant compactifications of $\mathbb G_a^n$.   
Another application can be found in \cite{goro-ramin-t}, where
this theorem plays an important role in an implementation of ideas from ergodic theory (mixing) 
in a proof of Manin's conjecture for equivariant compactifications of 
$G\backslash G^n$, where $G$ is an absolutely simple linear algebraic group, 
acting diagonally on $G^n$.

\

In Section~\ref{sect:general} we recall basic properties of the invariants $a(X,L)$ and $b(X,L)$.
After discussing balanced line bundles 
in Section~\ref{sect:balanced}, we turn to del Pezzo surfaces in Section~\ref{sect:pezzo}. 
In Section~\ref{sect:equi} we study the geometry of equivariant compactifications of 
homogeneous spaces and prove Theorem~\ref{thm:main}. 
In Section~\ref{sect:toric} we investigate balanced line bundles on toric varieties, in the context of the 
Minimal Model Program.

\

\noindent
{\bf Acknowledgments.} 
The authors would like to thank Brian Lehmann
for answering our questions and making useful comments.
The second author also would like to thank Fedor Bogomolov and Ilya Karzhemanov
for useful discussions.
The first author was supported by National Science Foundation Grant 0901645, 0968349, and 1148609
the third author was supported by National Science Foundation
Grants 0901777, 0968318, and 1160859.

\section{Generalities}
\label{sect:general}

\begin{defi}
\label{defi:cone}
Let $V$ be a finite dimensional vector space over $\bR$.
{\it A closed convex cone} $\Lambda \subset V$
is a closed subset which is closed under
linear combinations with non-negative real coefficients.
{\it An extremal face} $F \subset \Lambda$ is a closed convex subcone of $\Lambda$ such that if
$u, v \in \Lambda$ and $u+v \in F$ then $u, v \in F$.
{\it A supporting function} is a linear functional $\sigma : V \ra \bR$
such that $\sigma \geq 0$ on $\Lambda$.
A face of the form
$$
F' = \{ \sigma = 0\} \cap \Lambda
$$
is called {\it a supported face}.
A supported face is an extremal face,
but the converse is not true, in general.
The converse does hold when $\Lambda$ is locally finitely generated in a neighborhood of $F$,
i.e., there exist finitely many linear functionals
$$
\lambda_i: V \ra \bR
$$
such that $\lambda_i(v)> 0$ for $v \in F \setminus \{0\}$ and
$$
\Lambda \cap \{ v : \lambda_i(v) \geq0 \text{ for any $i$}\},
$$
is finitely generated.
Note that when $\Lambda$ is {\it strict}, i.e.,
does not contain a line, then $\{0\}$ is a supported face.
\end{defi}

We work over an algebraically closed field of characteristic zero.  
A variety is an integral separated scheme over this field.  
Let $X$ be a smooth projective variety.
We use 
$$\Lambda_{\eff}(X) \subset \NS(X,\bR) \subset \rH^2(X,\bR)$$ 
to denote the pseudo-effective cone, i.e., 
the closure of effective $\bQ$-divisors on $X$ in the real N\'eron-Severi
group $\NS(X,\bR)$.  Another
common notation in the literature is $\overline{\mathrm{NE}}^1(X)$.  
Note that the pseudo-effective cone is strictly convex \cite[Prop.~1.3]{BFJ}.
Let $\Lambda^{\circ}_{\eff}(X)$ denote the interior of the pseudo-effective cone;
a divisor class $D$ on $X$ is {\em big} if $[D] \in \Lambda^{\circ}_{\eff}(X)$.
We denote the dual cone of the cone of pseudo-effective divisors by $\overline{\text{NM}}_1(X)$.
This is the closure of the cone generated by movable curves (\cite{BDPP}.)

A {\em rigid effective} divisor is a reduced divisor $D\subset X$ 
such that
$$
\mathrm H^0(\cO_X(nD)) = 1\quad \forall n\ge 1. 
$$
If $D$ is rigid with irreducible components $D_1,\ldots,D_r$ then 
$$
\mathrm H^0(\cO_X(n_1D_1+\ldots+n_rD_r))=1 \quad \forall n_1,\ldots,n_r \ge 1
$$
and 
\begin{equation} \label{eqn:rigid}
\mathrm{span}(D_1,\ldots,D_r)\cap \Lambda^{\circ}_{\eff}(X)=\emptyset.
\end{equation}

\begin{defi} \label{defi:adef}
Assume that 
$L$ is a big line bundle on $X$.
The {\em Fujita invariant} is defined by
\begin{eqnarray*}
a(X,L) &=&  \inf \{a \in \bR: aL+K_X \text{ effective }\}  \\
       &=&  \min \{a \in \bR: a[L]+[K_X] \in \Lambda_{\eff}(X)\}.
\end{eqnarray*}
\end{defi}
Note that the Fujita invariant is positive if and only if $K_X$ is not 
pseudo-effective.  
The invariant 
$$
\kappa\epsilon(X,L)= -a(X,L)
$$ 
was introduced and studied by Fujita 
under the name {\em Kodaira energy} \cite{fujita-pure}, \cite{fujita-manu}, \cite{fujita-log}, \cite{fujita-three}. 
A similar invariant 
$$
\sigma(X,L) = \dim(X)+1-a(X,L)
$$ 
appeared in \cite{sommese} under the name {\em spectral value}. 

\begin{rema} \label{rema:uniruled} 
A smooth projective variety $X$ is uniruled if and only if
$K_X$ is not pseudo-effective \cite{BDPP}, \cite[Cor.~11.4.20]{LazII}.
\end{rema}

The following result was conjectured by Fujita
and proved by Batyrev for threefolds and \cite[Cor. 1.1.7]{BCHM} in general.
\begin{theo}
\label{theo:alpha_rational}
Let $X$ be projective with Kawamata log terminal singularities
such that $K_X$ is not pseudo-effective, 
and $L$ an ample line bundle on $X$.
Then $a(X,L)$ is rational.
\end{theo}
However, this property can fail when $L$ is big but not ample,
and the following example going back to Cutkosky \cite{Cut86} was suggested to us 
by Brian Lehmann:
\begin{exam}\cite[Example 4.9]{L11}
\label{exam:nonrational}
Let $Y$ be an abelian surface with Picard rank at least 3.
The cone of nef divisors and the cone of pseudo-effective divisors coincide, 
and the boundary of these cones is circular.
Let $N$ be a line bundle on $Y$ such that $-N$ is ample and
$X := \bP(\mathcal O \oplus \mathcal O(N))$. 
Let $\pi :X\ra Y$  denote the projection morphism and
$S \subset X$ the section corresponding
to the quotient map $\mathcal O \oplus \mathcal O(N) \ra \mathcal O(N)$.
Every divisor on $X$ is linearly equivalent to $tS + \pi^*D$
where $D$ is a divisor on $Y$. In particular, $K_X$
is linearly equivalent to $-2S + \pi^*N$.
The cone of pseudo-effective divisors $\Lambda_{\eff}(X)$ is generated by
$S$ and $\pi^*\Lambda_{\eff}(Y)$.

Consider a big $\bQ$-divisor $L = tS + \pi^*D$,
where $t > 0$ and $D$ is a big $\bQ$-divisor on $Y$.
If $t$ is sufficiently large,
then $a(X, L)$ is defined by $a[D] + [N] \in \partial \Lambda_{\eff}(Y)$.
However, the boundary of $\Lambda_{\eff}(Y)$ is circular,
and $a(X,L)\notin \bQ$, in general.
\end{exam}

From the point of view of Manin's conjecture,
the global geometric invariants involved in its formulation should be functorial for birational
transformations, and indeed this holds for the Fujita invariant:

\begin{prop}
\label{prop:alpha}
Let $\beta:\tilde{X}\ra X$ be a birational morphism of projective varieties,
where $\tilde{X}$ is smooth and $X$ has canonical singularities.  Assume $K_{X}$ is
not pseudo-effective 
and $L$ is big.  Setting
$\tilde{L}=\beta^*L$, we have
$$a(X,L)=a(\tilde{X},\tilde{L}).$$
\end{prop}

\begin{proof}
Since $X$ has canonical singularities, we have
$$K_{\tilde{X}}=\beta^*K_X+\sum_i d_i E_i,$$
where the $E_i$ are the irreducible exceptional divisors and the 
$d_i$ are nonnegative rational numbers.
It follows that for integers $m,n\ge 0$ we have
\begin{eqnarray*}
\Gamma(\cO_X(mK_X+nL))&=&
\Gamma(\cO_{\tilde{X}}(mK_{\tilde{X}}-\sum_i \lfloor md_i \rfloor E_i+n\tilde{L}))\\
  &=&
\Gamma(\cO_{\tilde{X}}(mK_{\tilde{X}}+n\tilde{L})), 
\end{eqnarray*}
where the second equality reflects the fact that allowing poles in the exceptional locus
does not increase the number of global sections.
In particular, effective divisors supported in the exceptional locus of $\beta$ are rigid.
It follows from the assumption that no multiple of $K_{\tilde{X}}$ is effective
and that $a(\tilde{X}, \tilde{L}) \geq 0$.
Definition~\ref{defi:adef} gives $a(\tilde{X}, \tilde{L}) = a(X, L) > 0$.
\end{proof}

\

Next, we discuss the second geometric invariant appearing
in Manin's conjecture.

\begin{defi}
\label{defi:b_defi}
Let $X$ be a projective variety with only $\bQ$-factorial terminal singularities
such that $K_X$ is not pseudo-effective.
Let $L$ be a big line bundle on $X$.
Define
$$\begin{array}{rcl}
b(X,L)&=&\text{the codimension of the minimal supported face of} \\
& & \Lambda_{\eff}(X) 
        \text{ containing the $\bR$-divisor $a(X,L)L + K_X$.}  
\end{array}
$$
\end{defi}

This definition is relatively easy to grasp when $\Lambda_{\eff}(X)$ is finitely
generated, which holds in a number of cases:
\begin{itemize}
\item{
A projective variety $X$ is {\em log Fano} if there exists an effective $\bQ$-divisor
$\Delta$ on $X$ such that $(X,\Delta)$ is divisorially log terminal (see \cite[p.~424]{BCHM})
and $-(K_X+\Delta)$ is ample.  
If $X$ is log Fano then the Cox ring of $X$ is finitely generated \cite[Cor.~1.3.2]{BCHM}, so in particular,
$\Lambda_{\eff}(X)$ is finite rational polyhedral and generated by effective divisors.  }
\item{Let $X$ be a smooth projective variety that is toric or an equivariant compactification
of the additive group $\bG_a^n$.  Then $\Lambda_{\eff}(X)$ is generated
by boundary divisors, i.e., irreducible components of the complement of the open 
orbit, \cite[Thm.~2.5]{HT}, \cite[Prop.1.2.11]{BT-IMRN}. }
\end{itemize}

Let $X$ be a smooth projective variety with $\Lambda_{\eff}(X)$ generated by a finite number
of effective divisors and 
$\Pic(X)_\bQ = \NS (X, \bQ)$.
Since each irreducible rigid effective divisor on $X$ is a generator of $\Lambda_{\eff}(X)$
(cf.~(\ref{eqn:rigid})), we have
$$
Z:=\cup_{\text{rigid effective}} D
$$
is a Zariski closed proper subset of $X$.

One of the reasons for adopting the terminology of supported faces in the definition of $b(X, L)$ is
to simplify the verification of its birational invariance:

\begin{prop}
\label{prop:Beta}
Let $X$ be a $\bQ$-factorial terminal projective variety
such that $K_X$ is not pseudo-effective and $\beta : \tilde{X} \ra X$ a smooth resolution.
Let $L$ be a big line bundle on $X$ and put $\tilde{L} = \beta^*L$. Then 
$$
b(X, L) = b(\tilde{X}, \tilde{L}).
$$
\end{prop}
\begin{proof}
Let $F$ be the minimal supported face of $\Lambda_\eff(X)$
containing an $\bR$-divisor $a(X, L)L + K_X$
and $\tilde{F}$ be the minimal supported face of $\Lambda_\eff(\tilde{X})$
containing $a(\tilde{X}, \tilde{L})\tilde{L} + K_{\tilde{X}}$.
The vector spaces generated by $F$ and $\tilde{F}$ will be denoted 
by $V_F$ and $V_{\tilde{F}}$, respectively.
There exists a nef cycle $\xi \in \overline{\text{NM}}_1(\tilde{X})$
such that 
$$
\tilde{F} = \{ \xi = 0\} \cap \Lambda_{\eff}(\tilde{X}).
$$
Let $E_1, \ldots, E_n$ be irreducible components of the exceptional locus of $\beta$.
The Negativity Lemma (\cite[Lemma 3.6.2]{BCHM})
implies that $\NS (\tilde{X})$ is a direct sum of $\beta^*\NS (X)$ and $[E_i]$'s.
Since $X$ has only terminal singularities, it follows from Proposition \ref{prop:alpha} that
$$
a(\tilde{X}, \tilde{L})\tilde{L} + K_{\tilde{X}} = \beta^*(a(X, L)L + K_X) + \sum_i d_i E_i,
$$
where the $d_i$'s are positive rational numbers.
This implies that $\tilde{F}$ contains $\beta^*(a(X, L)L + K_X)$ and the $E_i$'s.
Thus $a(X, L)L + K_{X}$ is contained in an extremal face
supported by a supporting function $\beta_*\xi$, i.e., 
$$
\{ \beta_* \xi = 0\} \cap \Lambda_{\eff}(X),
$$
so this supported face also contains $F$.
We get a well-defined injection
$$
\Phi : V_F \hookrightarrow V_{\tilde{F}}/(\sum_i \bR E_i).
$$

On the other hand, let $\eta \in \overline{\text{NM}}_1(X)$ be a nef cycle supporting $F$.
Consider a linear functional 
$
\tilde{\eta} : \NS(\tilde{X}) \ra \bR
$
defined by 
$$
\text{$\tilde{\eta} \equiv \eta$ on $\beta^*\NS(X)$ and $\tilde{\eta} \cdot E_i = 0$ for any $i$.}
$$
The projection from $\NS (\tilde{X})$ to $\beta^*\NS (X)$
maps pseudo-effective divisors to pseudo-effective divisors
so that $\tilde{\eta} \in \overline{\text{NM}}_1(\tilde{X})$.
Moreover,
$$
\tilde{\eta} \cdot (a(\tilde{X}, \tilde{L})\tilde{L} + K_{\tilde{X}}) =
\tilde{\eta} \cdot (\beta^*(a(X, L)L + K_X) + \sum_i d_i E_i) = 0,
$$
so that $\{\tilde{\eta} = 0\} \cap \Lambda_\eff(\tilde{X})$ contains $\tilde{F}$.
It follows that $\Phi$ is bijective and 
our assertion is proved.
\end{proof}

\begin{defi}
Let $X$ be a uniruled projective variety with big line bundle $L$.
We define
$$a(X,L)=a(\tilde{X},\beta^*L), \quad b(X,L)=b(\tilde{X},\beta^*L)$$
where $\beta:\tilde{X} \ra X$ is some resolution of singularities.  
\end{defi}
Note that $K_{\tilde{X}}$ is not pseudo-effective by Remark~\ref{rema:uniruled};
Propositions~\ref{prop:alpha} and \ref{prop:Beta} guarantee the invariants
are independent of the choice of resolution.

\begin{exam}[The anticanonical line bundle]
\label{exam:biganti}
Let $X$ be a projective variety with only $\bQ$-factorial terminal singularities.
As in the smooth case, the cone $\Lambda_{\eff}(X)$ of pseudo-effective divisors is strict.
When the anticanonical class $-K_X$ is big, we have
$$
a(X, -K_X) = 1, \quad b(X, -K_X) = \rk \, \NS(X).
$$
Let $\beta:\tilde{X} \ra X$ be a smooth resolution
and $E_1, \ldots, E_n$ the irreducible components of the exceptional locus;
we have
$$
K_{\tilde{X}} = \beta^*K_X + \sum_i d_i E_i,
$$
where $d_i\in \bQ_{>0}$, for all $i$. 
Hence the minimal extremal face containing $-\beta^*K_X + K_{\tilde{X}}$
contains a simplicial cone $F : = \oplus_i \bR_{\geq 0}[E_i]$.
The fact that $F$ is a supported face follows from 
\cite[Theorem 3.19]{B04}, which asserts that the pseudo-effective cone is
locally polyhedral in this region, generated by the prime exceptional divisors.
Hence we may compute
$$
b(\tilde{X}, -\beta^*K_X) = \rk \, \NS(\tilde{X}, \bR) - n = \rk \, \NS(X, \bR) = b(X, -K_X).
$$
\end{exam}

\begin{rema}
There exist projective bundles over curves of arbitrary genus $g>0$ with big anticanonical divisor
\cite[3.13]{KMM}.  When $g>1$ these cannot have potentially dense
rational points.  
\end{rema}

\begin{exam}
\label{exam:numdimzero}
Example~\ref{exam:biganti} can be generalized as follows:
Let $X$ be a smooth projective variety and $D = \sum_i e_iE_i$ an effective $\bR$-divisor
whose numerical dimension $\nu (D)$ is zero (see \cite[Theorem 1.1]{L11B} for definitions).
The minimal extremal face containing $D$
contains $F = \oplus_i \bR_{\geq 0}[E_i]$,
and we claim that $F$ is a supported face of $\Lambda_\eff(X)$.
First we prove that $F$ is an extremal face.
Let $u, v \in \Lambda_\eff(X)$ such that $u+v \in F$.
For any pseudo-effective numerical class $\alpha$, we denote
the negative part and the positive part 
of the divisorial Zariski decomposition of $\alpha$ by $N_\sigma(\alpha)$
and $P_\sigma(\alpha) = \alpha - [N_\sigma(\alpha)] \in \Lambda_\eff(X)$ respectively
(see \cite[Section 3]{L11B}). The assumption 
$\nu(D) = 0$ implies that 
$$
u+v \equiv N_\sigma(u+v) \leq N_\sigma(u) + N_\sigma(v).
$$
This implies that $N_\sigma(u) \equiv u$ and $N_\sigma(v) \equiv v$
so that $u, v \in F$.
Again by \cite[Theorem 3.19]{B04},
the cone $\Lambda_\eff(X)$ is locally rational polyhedral
in a neighborhood of $F$. Hence $F$ is a supported face.
\end{exam}

Little is known about the geometric meaning of the invariant $b(X, L)$, in general. 
Here we consider situations relevant for our applications to
equivariant compactifications of homogeneous spaces.

\begin{defi}
\label{defi:poly}
Let $X$ be a $\bQ$-factorial terminal and projective variety
and $D$ an $\bR$-divisor in the boundary of $\Lambda_{\eff}(X)$.  
We say $D$ is {\it locally rational polyhedral}
if there exist finitely many linear functionals
$$
\lambda_i:\NS(X,\bQ) \ra \bQ
$$
such that $\lambda_i(D)> 0$ and
$$
\Lambda_{\eff}(X) \cap \{ v : \lambda_i(v) \geq 0 \text{ for any $i$}\},
$$
is finite rational polyhedral and generated by effective $\bQ$-divisors.
In this case, the minimal extremal face $F$ containing $D$
is supported by a supporting function.
\end{defi}

\begin{theo}
\label{theo:adjoint}
Let $X$ be a $\bQ$-factorial terminal projective variety
such that $K_X$ is not pseudo-effective and
$L$ a big line bundle on $X$.
Suppose that $a(X, L)L +K_X$ has the form
$c(A + K_X + \Delta)$, where $A$ is an ample $\bR$-divisor,
$(X, \Delta)$ a Kawamata log terminal pair, and $c>0$.
Then $a(X, L)L +K_X$ is locally rational polyhedral and $a(X, L)$ is rational.
\end{theo}

\begin{proof}
The local finiteness of the pseudo-effective boundary
is proved in \cite[Proposition 3.3]{L11}
by using the finiteness of the ample models
\cite[Corollary 1.1.5]{BCHM}.
Moreover, Lehmann proved that
the pseudo-effective boundary is locally defined by movable curves.
The generation by effective $\bQ$-divisors follows from 
the non-vanishing theorem \cite[Theorem D]{BCHM}.
\end{proof}

In particular,
if $L$ is ample then
$a(X, L)L+K_X$ is locally rational polyhedral.
We have already seen in Example \ref{exam:nonrational} that 
the local finiteness is no longer true
if we only assume that $L$ is big.
However, there are certain cases
where the local finiteness of $a(X, L)L + K_X$ still holds for any big line bundle $L$:

\begin{exam}[Surfaces]
\label{exam:surface}
Let $X$ be a smooth projective surface
such that $K_X$ is not pseudo-effective.
Let $L$ be a big line bundle on $X$.
Suppose that $D = a(X, L)L+K_X$ is a non-zero pseudo-effective divisor.
Then $D$ is locally rational polyhedral.
We consider the Zariski decomposition of $D = P + N$,
where $P$ is a nef $\bR$-divisor and $N$ is the negative part of $D$.
The boundary of $\Lambda_{\text{eff}}(X)$ is locally rational
polyhedral away from the nef cone (see \cite[Theorem 3.19]{B04} and \cite[Theorem 4.1]{B04}).
Thus, if $N$ is non-zero, then our assertion follows.
Suppose that $N$ is zero. Since $a(X, L)L \cdot P + K_X \cdot P = D\cdot P = 0$
and $L\cdot P > 0$, we have $K_X \cdot P < 0$.
Thus our assertion follows from Mori's cone theorem.
In particular, $a(X, L)$ is a rational number.
\end{exam}

\begin{exam}[Equivariant compactifications of the additive groups]\
\label{exam:additivegroup}
Let $X$ be a smooth projective equivariant compactification
of the additive group $\mathbb G_a^n$.
Then $\Lambda_{\text{eff}}(X)$ is a simplicial cone
generated by boundary components, by 
\cite[Theorem 2.5]{HT}.
However, this cannot be explained from
Theorem \ref{theo:adjoint}.
Indeed, consider the standard embedding of $\mathbb G_a^3$ into $\bP^3$:
$$
\mathbb G_a^3 \ni (x, y, z) \mapsto (x:y:z:1) \in \bP^3.
$$
This is an equivariant compactification,
and the group action fixes every point on the boundary divisor $D$, a hyperplane section.
Let $X$ be an equivariant blow up of 12 generic points on a smooth cubic curve
in  $D$.
Write $H$ for the pullback of the hyperplane class and $E_1, \ldots, E_{12}$
for the exceptional divisors.
Consider 
$$
L = 4H -E_1 - \cdots -E_{12}.
$$
Then $L$ is big and nef, but not semi-ample
(see \cite[Section 2.3.A]{LazI} for more details).
In particular, the section ring of $L$ is not finitely generated.
On the other hand, consider
$$
\Lambda_{\text{adj}}(X) = \{ \Gamma \in \Lambda_{\text{eff}}(X) \mid \Gamma = c(A + K_X + \Delta) \},
$$
where $A$ is an ample $\bR$-divisor, $(X, \Delta)$ a Kawamata log terminal pair,
and $c$ a positive number.
Then $\Lambda_{\text{adj}}(X)$ forms a convex cone.
The existence of non-finitely generated divisors and 
\cite[Corollary 1.1.9]{BCHM} imply that $\Lambda_{\text{adj}}(X) \subsetneqq 
\Lambda_{\text{eff}}(X)$.
\end{exam}

It is natural to expect that 
the invariant $b(X, L)$ is related to the canonical fibration
associated to $a(X, L)L+K_X$.
A sample result in this direction is:

\begin{prop}
\label{prop:fibration}
Let $X$ be a smooth projective variety such that $K_X$ is not pseudo-effective.
Let $L$ be a big line bundle
and assume that 
$
D = a(X, L)L + K_{X}
$
is locally rational polyhedral and semi-ample. 
Let $\pi : X \ra Y$ be the semi-ample fibration of $D$. Then
$$
b(X, L) = \rk \,\NS(X) - \rk \,\NS_{\pi}(X),
$$
where $\NS_{\pi}(X)$ is the lattice generated by $\pi$-vertical divisors, i.e., 
divisors $M\subset X$ such that $\pi(M)\subsetneq Y$
\end{prop}
\begin{proof}
Let $F$ be the minimal extremal face of $\Lambda_{\text{eff}}(X)$
containing $D = a(X, L)L+K_X$
and $V_F$ the vector space generated by $F$.
We claim that $V_F = \NS_\pi(X)$.
Let $H$ be an ample $\bQ$-divisor on $Y$
such that $\pi^*H = D$.
Let $M$ be a $\pi$-vertical divisor on $X$.
Then for sufficiently large $m$,
there exists an effective Cartier divisor $H'$
such that $mH \sim H'$
and the support of $H'$ contains $\pi(M)$.
Thus $mD = m \pi^*H \sim \pi^*H' \in F$
and the support of $\pi^*H'$ contains $M$.
We conclude that $M \in F$,
and this proves that $\NS_\pi(X) \subset V_F$.
Next, let $X_y$ be a general fiber of $\pi$
and $C \subset X_y$ a movable curve on $X$
such that $[C]$ is in the interior of $\overline{\text{NM}}_1(X_y)$.
Then 
$$
F_C = \{ [C] = 0\} \cap \Lambda_{\text{eff}}(X),
$$
is an extremal face containing $D$.
The minimality implies $F \subset F_C$.
On the other hand, the local rational finiteness of $D$
implies that there exist effective $\bQ$-divisors $D_1, \ldots, D_n \in F$
such that $D_1, \ldots, D_n$ form a basis of $V_F$.
Since $D_1\cdot C = \cdots = D_n\cdot C = 0$,
the supports of $D_i$'s are $\pi$-vertical.
Hence it follows that $V_F \subset \NS_\pi(X)$.
\end{proof}

\begin{rema}
\label{rema:relative}
When $L$ is ample, it follows from
\cite[Lemma 3.2.5]{KMM87} that the codimension of the minimal extremal face
of nef cone containing $D$ is equal to the relative Picard rank $\rho(X/Y)$.
\end{rema}

\noindent
In Section \ref{sect:toric}, we explore this further in the case of toric varieties.

\section{Balanced line bundles}
\label{sect:balanced}

\begin{defi}
\label{defn:balanced}
Let $X$ be a uniruled projective variety, $L$
a big line bundle on $X$, and $Y\subsetneq X$ an irreducible
uniruled subvariety.  $L$ is
{\em weakly balanced with respect to $Y$} if
\begin{itemize}
\item{$L|_Y$ is big;}
\item{$a(Y,L|_Y) \le a(X,L)$;}
\item{if $a(Y,L|_Y)=a(X,L)$ then $b(Y,L|_Y) \le  b(X,L)$.}
\end{itemize}
It is {\em balanced with respect to $Y$} 
if it is weakly balanced and one of the two inequalities is strict.

$L$ is {\em weakly balanced} (resp.~{\em balanced}) {\em on $X$} if there exists a Zariski closed subset 
$Z\subsetneq X$ such that $L$ is weakly balanced (resp.~balanced) with respect to every $Y$ 
not contained in $Z$. The subset $Z$ will be called {\em exceptional}.  
\end{defi}

\begin{rema}
The restriction to uniruled subvarieties is quite natural:  If $Y$
is a smooth projective variety that is {\em not} uniruled and
$Y\ra X$ is a morphism such that $L|_Y$ is big, then 
$a(Y,L|_Y) \le 0 < a(X,L)$
(see Remark~\ref{rema:uniruled}).  
\end{rema}

We first explore these properties for projective homogeneous spaces:

\begin{prop}
\label{prop:gp}
Let $G$ be a connected semi-simple algebraic group, $P\subset G$ a parabolic subgroup and 
$X=P\backslash G$ the associated generalized flag variety. 
Let $L$ be a big line bundle on $X$. We have:
\begin{itemize}
\item{if $L$ is not proportional to $-K_X$ then $L$ is weakly balanced
but not balanced;}
\item{if $L$ is proportional to $-K_X$ then $L$ is balanced with respect to 
any smooth subvariety  $Y \subset X$.}
\end{itemize}
\end{prop}

\begin{proof}
For generalized flag varieties, the nef cone and the pseudo-effective cone coincide
so that $L$ is ample.
Moreover, the nef cone
of a flag variety is finitely generated by semi-ample line bundles.
Also note that since every rationally connected smooth proper variety is simply connected,
all parabolic subgroups are connected.

Assume that $L$ is not proportional to the anticanonical bundle,
i.e., $D = a(X, L)L + K_X$ is a non-zero effective $\bQ$-divisor.
Let $\pi : X \ra X'$ be the semi-ample fibration of $D$.
Then $X'$ is also a $G$-variety
so that there exists a parabolic subgroup $P' \supset P$
such that $X' = P'\backslash G$ and $\pi$ is the natural projection map.
We have the following exact sequence:
$$
0 \ra \Pic  (P'\backslash G)_\bQ \ra\Pic (P\backslash G)_\bQ
\ra \Pic (P\backslash P')_\bQ \ra 0.
$$
Indeed, the surjectivity follows from \cite[Proposition 3.2(i)]{KKV}.
Then the exactness of other parts follows from \cite[Lemma 3.2.5]{KMM87}.
Let $W$ be a fiber of $\pi$.
Then $a(X, L) = a(W, L|_W)$ since $K_X|_W = K_W$.
The exact sequence and Remark \ref{rema:relative}
imply that 
$$
b(X, L) = \rho (X/X') = \rk \, \Pic (W) = b(W, L|_W).
$$
Thus $L$ is not balanced with respect to any fiber of $\pi$.

Let $L= -K_X$ and $Y \subset X$ a smooth subvariety.
Let $\mathfrak g$ be the Lie algebra of $G$.
For any $\partial \in \mathfrak g$, we can construct 
a global vector field $\partial^X$ on $X$ such that
for any open set $U \subset X$ and any $f \in \mathcal O_X(U)$,
$$
\partial^X(f) (x) = \partial_g f(x\cdot g)|_{g=1}.
$$
It follows that the normal bundle $\mathcal N_{Y/X}$ is globally generated.
We conclude that
$$
-K_X|_Y +K_Y = \det (\mathcal N_{Y/X}) \in \Lambda_{\eff}(Y),
$$
so that $a(Y, -K_X|_Y) \leq a(X, -K_X) = 1$.

Suppose that $a(Y, -K_X|_Y) = a(X, -K_X) = 1$.
Our goal is to prove that 
$$
b(Y, -K_X|_Y) < b(X, -K_X) = \rk \, \NS (X).
$$

First we assume that $\det (\mathcal N_{Y/X})$ is trivial
so that $\mathcal N_{Y/X}$ is the trivial vector bundle of rank $r = \mathrm{codim}(Y, X)$.
The above construction of vector fields defines a surjective map:
$$
\varphi : \mathfrak g \ra \rH^0 (Y, \mathcal N_{Y/X}).
$$
We may assume that $e = P \in Y$ so that the Lie algebra $\mathfrak p$ of $P$
is contained in the kernel of $\varphi$.
Consider the Hilbert scheme $\mathrm{Hilb}(X)$ and 
note that $\rH^0 (Y, \mathcal N_{Y/X})$ is naturally isomorphic to
the Zariski tangent space of $\mathrm{Hilb} (X)$ at $[Y]$.
Consider the morphism:
$$
\pi : G \ni g \mapsto [Y\cdot g] \in \mathrm{Hilb}(X).
$$
Since $Y$ is Fano, $\mathrm{H}^1(Y, \mathcal N_{Y/X}) = 0$,  $\mathrm{Hilb}(X)$ is smooth at $[Y]$, and
$$
\dim_{[Y]} (\mathrm{Hilb}(X)) = r.
$$
Moreover, since $\varphi$ is surjective, $\pi$ is a smooth morphism
and $\pi (G)$ is a smooth open subscheme in $\mathrm{Hilb}(X)$.
Let $\mathcal H$ be the connected component of $\mathrm{Hilb}(X)$ containing $[Y]$ and
$P' = \mathrm{Stab}(Y)$.
Since the kernel of $\varphi$ contains $\mathfrak p$,
we have $P \subset P'$.
This implies that $\pi(G) = P'\backslash G$ is open and closed
so that 
$$
\mathcal H = \pi(G) = P' \backslash G.
$$
In particular, $\dim(G) - \dim(P') = r$,
so the kernel of $\varphi$ is exactly equal to the Lie algebra $\mathfrak p'$ of $P'$.
Consider the universal family $\mathcal U \subset X \times \mathcal H$ on $\mathcal H$.
It follows that $G$ acts on $\mathcal U$ transitively,
and we conclude that $\mathcal U = P \backslash G$ and $Y = P \backslash P'$.
Our assertion follows from the exact sequence which we discussed before.

In the general case, we still know that $\mathcal N_{Y/X}$ and its
determinant are globally generated.
Consider the resulting fibration $Y\ra W$ with generic fiber $Y_w$,
which is smooth.  Note that $\mathcal N_{Y_w/X}$ is trivial, as 
$\det(\mathcal N_{Y/X})|_{Y_w}$ and $\mathcal N_{Y_w/Y}$ are both trivial.
The above construction shows 
that $Y_w$ is the fiber of a fibration $\rho : X \ra B$ with $W\subset B$;
$Y$ is the pullback of $W$.
Theorem \ref{theo:adjoint} and Proposition \ref{prop:fibration} imply that
$$
b(Y, -K_X|_Y) = \rk \, \NS (Y) - \rk \, \NS_\rho (Y).
$$
The restriction map
$$
\Phi : \NS(Y) / \NS_\rho (Y) \ra \NS (Y_w),
$$
is injective;
this follows from \cite[Lemma 3.2.5]{KMM87}.
Hence 
$$
b(Y, -K_X|_Y) \leq b(Y_w, -K_X|_{Y_w}) < b(X, -K_X).
$$
\end{proof}

The following proposition is straightforward:

\begin{prop}
\label{prop:divisor}
Let $X$ be a smooth Fano variety of Picard rank one and $Y\subset X$ an irreducible smooth effective divisor. 
Then $-K_X$ is balanced with respect to $Y$. 
\end{prop}

\begin{proof}
For smooth divisors $Y\subset X$ the claim follows from adjunction formula: 
$$
-K_X|_Y + K_Y \in \Lambda_{\eff}(Y)^\circ,
$$
because $Y$ is an ample divisor on $X$.
Thus we obtain
$$
a(Y,-K_X|_Y) < a(X,-K_X) = 1.
$$
\end{proof}

However, this may fail when $Y$ is singular:

\begin{exam}[Mukai-Umemura 3-folds, \cite{MU83}]
\label{exam:MU}
Consider the standard action of $\mathrm{SL}_2$ on $V = \bC x \oplus \bC y$.
Let $R_{12} = \mathrm{Sym}^{12}(V)$ be a space of homogeneous polynomials of degree 12 in two variables
and $f \in R$ a general form.
Let $X$ be the Zariski closure of the $\mathrm{SL}_2$-orbit $\mathrm{SL}_2 \cdot [f] \subset \bP(R_{12})$.
Then $X$ is a smooth Fano 3-fold of index 1 with $\Pic (X) = \bZ$, for general $f$.
The complement of the open orbit $\mathrm{SL}_2 \cdot [f]$ is an irreducible divisor
$$
D= \overline{\mathrm{SL}_2 \cdot [x^{11}y]} = \mathrm{SL}_2 \cdot [x^{11}y] \cup \mathrm{SL}_2 \cdot [x^{12}],
$$
a hyperplane section on $\bP(R_{12})$ whose class generates $\Pic (X)$.
Furthermore, $D$ is the image of $\bP^1 \times \bP^1$ by a linear series of bidegree $(11, 1)$,
which is injective, an open immersion outside of the diagonal, but not along the diagonal.
In particular, $D$ is singular along the diagonal.

Let $\beta : \tilde{D} \ra D$ be the normalization of $D$ which is isomorphic to $\bP^1 \times \bP^1$.
Then $-\beta^*K_X|_{\tilde{D}}$ is a line bundle of bidegree $(11, 1)$
so that 
$$
a(D, -K_X|_D) = a(\tilde{D}, -\beta^*K_X|_{\tilde{D}}) = 2 > 1= a(X, -K_X).
$$
Thus Proposition \ref{prop:divisor} does not hold for $D$.
\end{exam}

\begin{rema}
\label{rema:a_surface}
We do not know whether Proposition \ref{prop:divisor} holds for
singular surfaces $Y$ in $\bP^3$.
\end{rema}

Some of simplest examples
of Fano threefolds fail to be balanced:

\begin{exam}
\label{exam:three}
Let $X\subset \bP^4$ be a smooth cubic threefold,
which is Fano of index 2.
The Picard group of $X$ is generated by the hyperplane class $L$.
By Proposition~\ref{prop:divisor}, $-K_X$ is balanced with 
respect to every smooth divisor on $X$.
Let $Y\subset X$ be a line. Note that $2L$ 
restricts to the anticanonical class on $X$ and $Y$,
and $b(Y,L|_Y)=b(X,L)=1$.
Thus $-K_X$ is weakly balanced, but not balanced, with respect to $Y$. 
Since the family of lines dominates $X$, $-K_X$ is not balanced on $X$.  

However, assume that $X$ is defined over a number field. 
The family of lines dominating $X$ are surfaces of general type, which 
embeded into their Albanese varieties. By Faltings' theorem, lines defined over a fixed
number field lie on a proper subvariety and cannot dominate $X$. 
\end{exam}

\begin{exam}
\label{exam:twobytwo}
Let $X \subset \bP^5$ denote a smooth complete intersection of two quadrics.
The anticanonical class $-K_X=2L$ where $L$ is the hyperplane class,
which generates the Picard group.
The variety $A$ parametrizing lines $Y \subset X$ is an abelian surface \cite[p.~779]{GH}.
Four lines pass through a generic point $x\in X$ \cite[p.~781]{GH}, so
these lines dominate $X$.  We have $a(X,L)=a(Y,L|Y)=2$ and
$b(X,L)=b(Y,L|Y)=1$ so $X$ is not balanced.    

Suppose that $X$ is defined over a number field $F$ with $X(F)$ Zariski dense;
fix a metrization $\cL$ of $L$.
Manin's formalism predicts the existence of an open set $X^{\circ} \subset X$
such that
$$\# \{x\in X^{\circ}(F): \sH_{\cL}(x) \le \sB \} \sim c \sB^2.$$
However, each line $Y \subset X$ defined over $F$ contributes
$$\# \{x\in Y(F): \sH_{\cL}(x) \le \sB \} \sim c'(Y,\cL) \sB^2.$$
Moreover, after replacing $F$ by a suitable finite extension
these lines are Zariski dense in $X$, because
rational points on abelian surfaces are potentially dense (see \cite[\S 3]{HT00}, for instance). 
\end{exam}

For Fano varieties of index one,
one might hope to use the Fujita invariant $a(X,L)$ to identify the exceptional locus 
$X \setminus X^\circ$.
However, this is quite non-trivial even in the following situation, 
considered in \cite{Debarre} (see also  \cite{LT10} and \cite{Be06}):

\begin{conj}[Debarre - de Jong conjecture]
\label{conj:dejong}
Let $X \subset \bP^n$ be a Fano hypersurface of degree $d \leq n$.
Then the dimension of the variety of lines is $2n - d -3$.
In particular, when $d=n$, for any line $C$, we have
$$
a(C, -K_X|_C) = 2 > 1 = a(X, -K_X).
$$
The conjecture predicts that the dimension of the variety of lines is $n-3$
so that lines will not sweep out $X$.
\end{conj}

The weakly balanced property may fail too:

\begin{exam}
\cite{BT-cubic}
Let $f,g$ be general cubic forms on $\bP^3$ and 
$$
X:=\{ sf+tg = 0\} \subset \bP^1\times \bP^3,
$$
the Fano threefold obtained by blowing up the base locus of the pencil. 
The projection onto the first factor
exhibits a cubic surface fibration
$$
\pi: X\ra \bP^1,  
$$
so that $-K_X$ restricts to $-K_Y$, for every smooth fiber $Y$ of $\pi$.
Thus 
$$
a(Y,-K_Y) = a(X,-K_X)=1.
$$ 
Furthermore, the N\'eron-Severi rank of a smooth fiber of $\pi$ is 7. 
On the other hand, by the Lefschetz theorem, we have $\rk\, \NS(X) = 2$ and
$$
7=b(Y,-K_Y) >b(X,-K_X)=2, 
$$
i.e., $-K_X$ is not weakly balanced on $X$. 

Let $Z$ be the union of singular fibers of $\pi$, $-K_Y$-lines in general smooth fibers $Y$,
 and the exceptional locus of the blow up to $\bP^3$. Note that
$-K_X$ is balanced with respect to every rational curve on $X$ which is not contained in $Z$.   
\end{exam}

\section{del Pezzo surfaces}
\label{sect:pezzo}

Let $X$ be a smooth projective surface with ample $-K_X$, i.e., a del Pezzo surface. 
These are classified by the degree of the canonical class $d:=(K_X,K_X)$.
Basic examples are $\bP^2$ and $\bP^1\times \bP^1$; more examples are obtained
by blowing up $9-d$ general points on $\bP^2$.  
We have
\begin{itemize}
\item $\rk\,\NS(X) = 10-d$;
\item for $1\le d\le 7$ the cone $\Lambda_{\rm eff}(X)$ is 
generated by classes of exceptional curves, i.e., smooth rational 
curves of self-intersection $-1$. 
\end{itemize} 

Let $L$ be a big line bundle on $X$. When is it balanced? 
The only subvarieties of $X$ on which we need to test the values of $a$ and $b$ 
are rational curves $C\subset X$, and $b(C,L|_C)=1$.

It is easy to characterize curves breaking the balanced condition for the Fujita invariant.
Let $Z$ be the union of exceptional curves, if $d > 1$. When $d=1$, let
$Z$ be the union of exceptional curves
plus singular rational curves in $|-K_X|$.

\begin{lemm}
\label{lemm:balance-a-surf}
Let $X$ be a del Pezzo surface of degree $d$, 
$C$ an irreducible rational curve with $(C,C)\neq -1$, and  
$L$ a big line bundle on $X$.
Then
\begin{equation}
a(C,L|_C) \le a(X,L),
\end{equation}
i.e., $L$ is weakly balanced on $X$ outside of $Z$.
\end{lemm}

\begin{proof}
If $C$ is an irreducible rational curve with $(-K_X,C)=1$ then $C\subseteq Z$. 
Indeed, if $(C,C)<0$, then $(C,C)=-1$, by adjunction, and $C$ is exceptional. 
On the other hand, if $C$ and $-K_X$ are linearly independent,
the Hodge index theorem implies that
$d(C,C)-1<0$, i.e., $(C,C)=-1$ or $0$. The second case is impossible since $(K_X,C)+(C,C)$ must be even. 
If $C$ and $-K_X$ are linearly dependent, then $d(C, C) -1 = 0$
so that $d = 1$ and $C$ is a singular rational curve in $|-K_X|$.
 
Let $C\subset X$ be a rational curve which is not in $Z$. 
After rescaling, we may assume that $a(X, L) = 1$, in particular, we do not assume 
that $L$ is an integral divisor. 
Writing $L+K_X=D$, where $D$ is an effective $\bQ$-divisor, and computing 
the intersection with $C$ we obtain
$$
(L,C) = (-K_X,C) + (D,C) \ge (-K_X,C).
$$
Since $C$ is not in $Z$,  $(-K_X,C)\ge 2$, i.e., $(L,C)\ge 2$. It follows that 
$$
(L,C) + \deg(K_{\tilde{C}}) = (L,C) -2 \ge 0,
$$  
where $\tilde{C}$ is the normalization of $C$, i.e., $a(C,L|_C)\le 1$, as claimed.  
\end{proof}

We proceed with a characterization of $b(X,L)$. 
Consider the Zariski decomposition
$$
a(X,L)L+K_X = P + E,
$$
where $P$ is a nef $\bQ$-divisor and $E=\sum_{i=1}^n e_iE_i$, $e_i \in \bQ_{>0}$, $(E_i,E_j)<0$. 
We have $(P,E)=0$. 
By basepoint freeness (see \cite[Theorem 3.3]{KM98}), 
$P$ is semi-ample and defines a semi-ample fibration
$$
\pi:X\ra B.
$$
We have two cases:

\

{\em Case 1.} $B$ is a point. 
Then 
$$
a(X,L)L+K_X= \sum_{i=1}^n e_iE_i
$$
is rigid, which implies that the classes $E_i$ are linearly independent in $\NS(X)$. In particular, 
$\oplus_i \bR_{\ge 0} E_i$ is an extremal face of $\Lambda_{\rm eff}(X)$, and in fact the minimal extremal face
containing $a(X,L)L+K_X$.
It follows that
$$
b(X,L) = \rk \, \NS(X) -n. 
$$

\

{\em Case 2.} $B$ is a smooth rational curve. 
Then the minimal extremal face containing 
$$
a(X,L)L+K_X= P + \sum_{i=1}^n e_iE_i
$$
is given by 
$$
\NS_{\pi}(X)\cap \Lambda_{\rm eff}(X) = \{ P = 0 \} \cap  \Lambda_{\rm eff}(X),
$$
where 
$\NS_{\pi}(X)\subset \NS(X)$ is the subspace 
generated by vertical divisors, i.e., divisors $D\subset X$ 
not dominating $B$. 
It follows that 
$$
b(X,L)= \rk \, \NS(X) - \rk\, \NS_{\pi}(X) = 1. 
$$

\begin{prop}
\label{prop:surfaces}
Let $X$ be a del Pezzo surface and $L$ a big line bundle on $X$. 
Then $L$ is balanced if and only if $a(X, L)L+K_X=D$, where $D$ is a rigid effective divisor. 
\end{prop}

\begin{proof}
Assume that $a(X,L)=1$.  
In {\em Case 1}, we must have
$$
L+K_X=D=\sum_{i=1}^n e_iE_i, \quad e_i>0,
$$
with $E_i$ disjoint exceptional curves. 
Assume that $L$ is not balanced so that $b(X,L)=1$.
Let $\pi:X\ra \bP^2$ be the blowdown of $E_1,\ldots, E_n$ and $h$ a hyperplane class on 
$\bP^2$. Then 
$$
L=-K_X+D = 3\pi^*h + \sum_{i=1}^n (e_i-1)E_i.
$$
Let $C$ be an irreducible rational curve which is not in $Z$.
If $C$ does not meet any of the $E_i$ then 
$$
(L,C) = (3\pi^*h, C) \ge 3 >2.
$$
If $C$ meets at least one of the $E_i$ then 
$$
(L,C) = (-K_X,C) +(D,C) >2
$$
since the first summand is $\ge 2$.
It follows that $a(C,L_C)<1$, i.e., $L$ is balanced, contradicting our assumption. 

In {\em Case 2}, we have
$$
L+K_X=D=P+ \sum_{i=1}^n e_iE_i, \quad e_i\ge 0,
$$
where $P$ is nef and $E_i$ are disjoint exceptional divisors. 
Let $\pi : X\ra \bP^1$ be the fibration induced by the semi-ample line bundle $P$. 
The general fiber $F$ of $\pi$ is a conic and 
$$
\rk \, \NS(X) - \rk \, \NS_{\pi}(X) =1.
$$  
We have $(F,F)=0$, $(-K_X,F) = 2$, and the class of $F$ is  proportional to $P$. Hence, for any such $F$, 
$$
a(F,L|_F) =a(X,L), \quad b(F,L|_F) =b(X,L)=1.
$$
Thus $L$ is not balanced.
\end{proof}

\section{Equivariant geometry}
\label{sect:equi}

Let $G$ be a connected linear algebraic group, $H\subset G$ a closed subgroup, and 
$X$ a projective equivariant compactification of $X^{\circ}:=H\backslash G$,
a quasi-projective variety \cite[Ch.~II]{borel}.
Applying equivariant resolution of singularities we may assume that $X$ is smooth and 
the boundary
$$
\cup_{\alpha\in \mathcal A} D_{\alpha} = X\setminus X^{\circ}
$$
is a divisor with normal crossings with irreducible components $D_{\alpha}$.  
If $H$ is a parabolic subgroup of a semi-simple group $G$,
then there is no boundary, i.e., $\mathcal A$ is empty,
and $H\backslash G$ is a generalized flag variety
which was discussed in Section~\ref{sect:balanced}.
Throughout, we will assume that 
$\mathcal A$ is not empty.

Let $\mathfrak X(G)^*$ be the group of algebraic characters of $G$ and 
$$
\mathfrak X(G, H)^{*}=\{\, \chi : G\ra \mathbb G_m \,|\,  \chi(hg)=\chi(g), \quad 
\forall h\in H \, \}
$$ 
the subgroup of characters whose restrictions to $H$ are trivial.
Let $\Pic^G(X)$ be the group of isomorphism classes of 
$G$-linearized line bundles on $X$ and $\Pic(X)$ the Picard group of $X$. 
For $L \in \Pic^G(X)$, the subgroup $H\subset G$ acts linearly on the fiber $L_x$ at $x 
= H \in H\backslash G$.
This defines a homomorphism 
$$
\Pic^G(X) \ra \mathfrak X(H)^{*}
$$
to characters of $H$. Let $\Pic^{(G, H)}(X)$ be the kernel of this map. 
We will identify line bundles and divisors with their classes in $\Pic(X)$. 

\begin{prop}
\label{prop:sequences}
Let $G$ be a connected linear algebraic group
and $H$ a closed subgroup of $G$.
Let $X$ be a smooth projective equivariant compactification of $X^\circ := H\backslash G$
with a boundary $\cup_{\alpha\in \mathcal A} D_{\alpha}$.
Then
\begin{enumerate}
\item we have an exact sequence
$$
0 \ra \mathfrak X(G, H)^{*} \ra \oplus_{\alpha \in \mathcal A} \bZ D_\alpha \ra \Pic(X) \ra \Pic (X^\circ) \ra 0;
$$
\item we have an exact sequence
$$
0 \ra \mathfrak X(G, H)^{*}_{\bQ} \ra \Pic^{(G, H)}(X)_{\bQ} \ra \Pic(X)_{\bQ};
$$
and the last homomorphism is surjective when 
$$
\mathfrak C(G, H) := \mathrm{Coker} (\mathfrak X(G)^* \ra \mathfrak X(H)^*)
$$
or equivalently, $\Pic (X^\circ)$, is finite.
\item we have a canonical injective homomorphism
$$
\Psi : \oplus_{\alpha \in \mathcal A}\bQ D_{\alpha} \hookrightarrow \Pic^{(G, H)}(X)_{\bQ};
$$
which is an isomorphism when 
$
\mathfrak C(G, H)
$
is finite.
\end{enumerate}
\end{prop}
\begin{proof}
The first statement is easy. 
The second assertion follows from \cite[Corollary 1.6]{GIT} and \cite[Proposition 3.2(i)]{KKV}.

For the last assertion:
Corollary 1.6 of \cite{GIT} implies that some multiple of $D_\alpha$ is $G$-linearizable.
We may assume that $G$ acts on the finite-dimensional vector space $\rH^0(X, \mathcal O_X(D_\alpha))$, 
via this $G$-linearization. Let $s_\alpha$ be the section corresponding to $D_\alpha$. 
Then $s_\alpha \in \rH^0(X, \mathcal O_X(D_\alpha))$ is an eigenvector of the action by $G$. 
After multiplying by a character of $G$, if necessary, we may assume that $s_\alpha$ 
is fixed by the action of $G$. 
We let $\Phi(D_\alpha)$ be this $G$-linearization.

Suppose that $\Phi(\sum_\alpha d_\alpha D_\alpha) = \mathcal O_X$, 
with trivial $G$-linearization, where $d_\alpha \in \bZ$. 
Then there exists a rational function $f$ such that
$$
\text{div} (f) = \sum_\alpha d_\alpha D_\alpha.
$$
We may assume that $f$ is a character of $G$ whose restriction to $H$ is trivial. 
By the definition of $\Phi$, the function $f$ must be fixed by the $G$-linearization. 
This implies that $f \equiv 1$.
When $\mathfrak C(G, H)$ is finite, the surjectivity of $\Phi$ follows from (1) and (2).
\end{proof}

\begin{prop}
\label{prop:kx}
Let $G$ be a connected linear algebraic group
and $H$ a closed subgroup of $G$.
Let $X$ be a smooth projective equivariant compactification of $X^\circ := H\backslash G$
with a boundary $\cup_{\alpha\in \mathcal A} D_{\alpha}$.
If $\mathfrak C(G, H)$ is finite, then the anticanonical divisor $-K_X$ is big.
\end{prop}

\begin{proof}
Let $\mathfrak g$, resp. $\mathfrak h$,  be the Lie algebra of $G$, resp. $H$.  
For any $\partial \in \mathfrak g$, there is a unique global vector field $\partial^X$ on $X$ such that
for any open set $U \subset X$ and any $\mathsf f \in \mathcal O_X(U)$,
$$
\partial^X(\mathsf f) (x) = \partial_g \mathsf f(x\cdot g)|_{g=1}.
$$
Let $\partial_1, \ldots, \partial_n \in \mathfrak g$ be a lift of a basis for $\mathfrak g/ \mathfrak h$.
Consider the following global section of the anticanonical bundle $\det(\mathcal T_X)$:
$$
\delta = \partial_1^X \wedge \cdots \wedge \partial_n^X.
$$
Note that this section is nonzero at $x = H \in H \backslash G = X^\circ$.
The proof of \cite[Lemma 2.4]{chambert-t02} implies that $\delta$ vanishes along the boundary. 
Hence 
$$
\text{div}(\delta) = \sum_\alpha d_\alpha D_\alpha + (\text{an effective divisor in $X^\circ$}),
$$
where $n_\alpha >0$.
When $\mathfrak C(G, H)$ is finite,
then $\Pic(X)_\bQ$ is generated by boundary components
so that every ample divisor can be expressed as a linear combination of $D_\alpha$'s.
This implies that $\sum_\alpha d_\alpha D_\alpha$ is big
so $\text{div}(\delta)$ is also big.
\end{proof}

From now on we consider the following situation: 
Let $H\subset M\subset G$ be connected linear algebraic groups.
Typical examples arise when $G$ is a unipotent group or a product of absolutely simple groups
and $H$ and $M$ are arbitrary subgroups such that $H\backslash M$ is connected.
Let $X$ be a smooth projective $G$-equivariant compactification of $H\backslash G$, 
and $Y$ the induced compactification of $H\backslash M$.

\begin{lemm}
\label{lemm:domin}
Let $\pi : X\ra X'$ be a $G$-equivariant morphism onto a projective equivariant compactification of $M\backslash G$.
Assume that the projection $G \ra M \backslash G$ admits a rational section.
Then
\begin{itemize}
\item $\pi(D_{\alpha})=X'$ if and only if $D_{\alpha}\cap Y\neq \emptyset$;
\item if $D_{\alpha}\cap Y\neq \emptyset$ then $D_{\alpha}\cap Y$ is irreducible;
\item  if  $D_{\alpha}\cap Y\neq \emptyset$ and  $D_{\alpha'}\cap Y\neq \emptyset$, for $\alpha\neq\alpha'$
then  $D_{\alpha}\cap Y\neq D_{\alpha'}\cap Y$.
\end{itemize}
\end{lemm}

\begin{proof}
We have the diagram

\

\centerline{
\xymatrix{
H\backslash G\ar[d] & \subset  & X\ar[d]_{\pi} & \supset & Y \ar[d]^{\pi}\\
M\backslash G       & \subset  & X'            & \supset &  M\cdot e = {\rm point}          
}
}

\

\noindent
The first claim is evident. To prove the second assertion, choose a rational section $\sigma : M \backslash G\dashrightarrow G$ of the projection $G \ra M\backslash G$.
We may assume that a rational section is well-defined at a point $M \in M \backslash G$.
Consider the diagram

\

\centerline{
\xymatrix{
D_{\alpha} \ar[d]_{\pi} & \supset & D_{\alpha}^{\circ} \ar[d]^{\pi}\\
          X'            & \supset & M\backslash G           
}
}

\

\noindent
where $D_{\alpha}^{\circ} = D_{\alpha} \cap \pi^{-1}(M \backslash G)$.
We define a rational map 
$$
\begin{array}{rcl}
\Psi : D_{\alpha}^{\circ} & \dashrightarrow &  D_{\alpha}\cap Y \\ 
               x          & \mapsto & x\cdot (\sigma\circ \pi(x))^{-1}.
\end{array}
$$
Since $\Psi$ is dominant, $D_{\alpha}\cap Y$ is irreducible. 
Since the $G$-orbit of $D_\alpha \cap Y$ is $D_\alpha^\circ$, the third claim follows. 
\end{proof}

\begin{rema}
\label{rema:section}
When $M$ is a connected solvable group,
then $G$ is birationally isomorphic to $M \times (M\backslash G)$
so that the projection $G \ra M \backslash G$ has a rational section.
See \cite[Corollary 15.8]{LAG}.
\end{rema}

\begin{theo}
\label{thm:subgroup}
Let 
$$
H\subset M\subset G
$$ 
be connected linear algebraic groups.
Let $X$ be a smooth projective $G$-equivariant compactification of $H\backslash G$
and $Y\subset X$ the induced compactification of $H\backslash M$. 
Let $L$ be a big line bundle on $X$.
Assume that
\begin{itemize}
\item the projection $G \ra M \backslash G$ admits a rational section;
\item and $D := a(X, L)L + K_X$ is a rigid effective $\bQ$-divisor.
\end{itemize}
Furthermore, assume that either
\begin{enumerate}
\item $\Lambda_\eff (X)$ is finitely generated by effective divisors; or
\item there exists a birational contraction map
$f: X \dashrightarrow Z$ contracting $D$,
where $Z$ is a normal projective variety.
\end{enumerate}
Then $L$ is balanced with respect to $Y$.
\end{theo}

\begin{proof}
Let $X'$ be any smooth projective equivariant compactification of $M\backslash G$.
We consider a $G$-rational map $\pi : X \dashrightarrow X'$ mapping
$$
\pi : G \ni g \mapsto Mg \in M \backslash G.
$$
After applying a $G$-equivariant resolution 
of the indeterminacy of the projection $\pi$ if necessary,
we may assume that $\pi$ is a surjective morphism
and $X$ is a smooth equivariant compactification of $H\backslash G$
with a boundary divisor $\cup_\alpha D_\alpha$.
Note that $Y$ is a general fiber of $\pi$ so that $Y$ is smooth.
Write the rigid effective $\bQ$-divisor $D = a(X, L)L + K_X$ by
$$
D = a(X, L)L + K_X = \sum_{i=1}^n e_i E_i,
$$
where $E_i$'s are irreducible components of $a(X, L)L + K_X$
and $e_i \in \bQ_{>0}$.
Our goal is to show that 
$$
(a(Y, L|_Y),b(Y, L|_Y)) < (a(X, L), b(X, L)).
$$
Since the $E_i$'s are rigid effective divisors,
they are boundary components.
This implies that 
$$
a(X, L)L|_Y + K_Y = (a(X, L)L + K_X)|_Y = \sum_{i=1}^n e_i E_i |_Y \in \Lambda_\eff(Y).
$$
It follows that
$$
a(Y, L|_Y) \leq a(X, L).
$$

Assume that $a(Y, L|_Y) = a(X, L) =: a$.
Let $F$ be the minimal supported face of $\Lambda_\eff(X)$
containing $D = aL + K_X = \sum e_i E_i$
and $V_F$ a vector subspace generated by $F$.
Either condition (1) or (2) guarantees that
$F$ is generated by $E_i$'s so that
$$
b(X, L) = \rk \, \NS(X) - n.
$$
(See Proposition~\ref{prop:Beta} and Example~\ref{exam:biganti}.)
Let $F'$ be the minimal supported face of $\Lambda_\eff(Y)$
containing 
$$
D|_Y = aL|_Y + K_Y = \sum e_iE_i|_Y.
$$
Let $V'$ be a vector subspace generated by all components of $E_i\cap Y$.
Since $F'$ contains all components of $E_i\cap Y$,
we have $b(Y, L|_Y) \leq \mathrm{codim}(V')$.
Consider the restriction map:
$$
\Phi : \NS (X) / V_F \ra \NS (Y) / V'.
$$
It follows from \cite[Proposition 3.2(i)]{KKV}, Lemma~\ref{lemm:domin},
and the exact sequence (1) in Proposition~\ref{prop:sequences}
that $\Phi$ is surjective.
On the other hand, 
$\pi^*\NS(X')$ is contained in the kernel of $\Phi$,
so $\Phi$ has the nontrivial kernel.
We conclude that
$$
b(Y, L|_Y) \leq \mathrm{codim}(V') < b(X, L).
$$
\end{proof}

\begin{rema}
\label{rema:numerical}
Conditions (1) and (2) can be replaced by
the condition: the numerical dimension $\nu(D)$ is zero (see \cite{L11B} for definitions).
\end{rema}

\begin{coro}
\label{coro:subgroup}
Let 
$
H\subset M\subset G
$ 
be connected linear algebraic groups
and $X$ a smooth projective equivariant compactification of $H\backslash G$
with big anticanonical bundle.
Let $Y\subset X$ be the induced compactification of $H\backslash M$. 
Assume that the projection $G \ra M \backslash G$ admits a rational section.
Then $-K_X$ is balanced with respect to $Y$.
\end{coro}

\begin{exam}
\label{exam:interest}
Let $G={\rm PGL}_2$, $M=B$, a Borel subgroup of $G$ and $H=1$. 
Let $X=\bP^3$ be the standard equivariant compactification of $G$
given by
$$
\mathrm{PGL}_2 \ni 
\begin{bmatrix}
       a & b\\
       c & d         
       \end{bmatrix}
 \mapsto [a : b : c : d] \in \bP^3,
$$
with boundary $D:=\{ad -bc=0\} = \bP^1\times \bP^1$.
Then $Y=\bP^2$; with boundary $D_Y=Y\setminus B = \ell_1\cup \ell_2$, a union of two intersecting lines.  
We have $X'=\bP^1$. The projection 
$$
\pi :X\dashrightarrow X'
$$
has indeterminacy along one of lines, say $\ell_1$. Blowing up $\ell_1$, we obtain a fibration
$$
\tilde{\pi} : \tilde{X}\ra \bP^1. 
$$
We have
$$
a(\tilde{X}, -K_{\tilde{X}})= a(\tilde{Y}, -K_{\tilde{X}}|_{\tilde{Y}})=a(\tilde{Y}, -K_{\tilde{Y}})=1.
$$
Every boundary component of $\tilde{X}$ dominates the base $\bP^1$, since the $G$-action is transitive on the base. 
Lemma~\ref{lemm:domin} shows that the number of boundary components of $\tilde{Y}$
is equal to the number of boundary components of $\tilde{X}$, which equals the rank of $\NS(\tilde{X}) = 2$.  
However, $\mathfrak X(B)^*=\bZ$, and in particular, the rank of the Picard group of $\tilde{Y}$ is 
one less than the number of boundary components, i.e., 
$$
b(\tilde{Y}, -K_{\tilde{X}}|_{\tilde{Y}}) = 1 < 2 = b(\tilde{X}, -K_{\tilde{X}}).
$$
\end{exam}

The existence of rational sections is important,
and the second statement in Lemma \ref{lemm:domin} is not true in general:
\begin{exam}
\label{exam:hilb}
Consider the standard action of $\mathrm{PGL}_3$ on $\bP^2$.
Let $\bP^5$ be the space of conics and consider
$$
X^\circ = \{ T = (C, [p_1, p_2, p_3]) \in \bP^5 \times \mathrm{Hilb}^{[3]}(\bP^2) 
\mid \text{$T$ satisfies $(*)$} \},$$
$$(*) : \text{$C$ is smooth, $p_i$'s are distinct, and $p_i \in C$}.$$
Let $X$ be the Zariski closure of $X^\circ$, it is the Hilbert scheme of conics with zero dimensional subschemes
of length 3,
and is a smooth equivariant compactification of
a homogeneous space $S_3\backslash \mathrm{PGL}_3$.
Consider a $\bP^2$-fibration $f : X \ra \mathrm{Hilb}^{[3]}(\bP^2)$, the fiber over  a general point
$Z \in \mathrm{Hilb}^{[3]}(\bP^2)$ is a $\bP^2$, parametrizing conics passing through $Z$.
The degenerate cases correspond to two lines passing through $Z$; these 
form three boundary components $l_i$ on $\bP^2$.
However, general points on these components are on the same $\mathrm{PGL}_3$-orbit,
so there exists an irreducible boundary divisor $D \subset X$ 
such that $D \cap f^{-1}(Z) = l_1 \cup l_2\cup l_3$.
In other words, there is a non-trivial monodromy action on $l_i$'s.
However, the monodromy action on the Picard group
is trivial, and the balancedness still holds with respect these fibers.
\end{exam}

\section{Toric varieties}
\label{sect:toric}

Manin's conjecture for toric varieties was settled by Batyrev and Tschinkel
via harmonic analysis on the associated adele groups in \cite{BT-0} and \cite{BT-general}.
Implicitly,  \cite{BT-general} established a version of balancedness.
Here, we will use MMP to determine balanced line bundles. We expect that these
techniques would also be applicable to some non-equivariant varieties.
We refer to \cite{fujino-2} for details concerning toric Mori theory,
though most of properties we use hold formally for Mori dream spaces.

We start by recalling basic facts regarding toric Mori theory, which also 
hold for all Mori dream spaces (see \cite[Section 1]{MDS}):

\begin{prop}[$D$-Minimal Model Program]
\label{prop:MMPwrtD}
Let $X$ be a $\bQ$-factorial projective toric variety and $D$ a $\bQ$-divisor.
Then the minimal model program with respect to $D$ runs, i.e.,
\begin{enumerate}
\item for any extremal ray $R$ of ${\mathrm{NE}}_1(X)$, 
there exists the contraction morphism $\varphi_R$;
\item for any small contraction $\varphi_R$ of a $D$-negative extremal ray $R$, 
the $D$-flip $\psi : X \dashrightarrow X^+$ exists;
\item any sequence of $D$-flips terminates in finite steps;
\item and every nef line bundle is semi-ample.
\end{enumerate}
\end{prop}
\begin{proof}
See Theorem 4.5, Theorem 4.8, Theorem 4.9, and Proposition 4.6 in \cite{fujino-2}.
\end{proof}

\begin{prop}[Zariski decomposition]
\label{prop:Zariski decomposition}
Let $X$ be a $\bQ$-factorial projective toric variety
and $D$ a $\bQ$-effective divisor.
ApplyING $D$-MMP we obtain a birational contraction map
$f : X \dashrightarrow X'$,
with nef proper transform $D'$ of $D$.
Consider a common resolution:

\

\centerline{
\xymatrix{
&   \tilde{X}\ar[dl]_{\mu} \ar[d]_{\nu} \\
X\ar@{-->}[r]_f  & X'    
}
}

\

\noindent
Then 
\begin{enumerate}
\item $\mu^*D = \nu^*D' + E$, where $E$ is a $\nu$-exceptional effective $\bQ$-divisor;
\item the support of $E$ contains all divisors contracted by $f$;
\item if $g : \tilde{X} \rightarrow Y$ is the semi-ample fibration associated to $\nu^*D'$,
then for any $\nu$-exceptional effective Cartier divisor $E'$, the natural map
$$
\mathcal O_{Y} \rightarrow g_*\mathcal O(E'),
$$
is an isomorphism.
\end{enumerate}

\end{prop}
\begin{proof}
The assertions (1) and (2) follow from the Negativity lemma (see \cite[Lemma 4.10]{fujino-2}).
Also see \cite[Theorem 5.4]{fujino-2}.
\end{proof}

The invariant $b(X, L)$ can be characterized in terms of Zariski decomposition
of $a(X, L)L + K_X$:
\begin{prop}
\label{prop:beta_MDS}
Let $X$ be a $\bQ$-factorial projective toric variety
and $D$ an effective $\bQ$-divisor on $X$.
Suppose that
\begin{enumerate}
\item $D = P+N$, where $P$ is a nef and $N \geq 0$;
\item let $g : X \ra Y$ be the semi-ample fibration associated to $P$.
For any effective Cartier divisor $E$ which is supported by $\textnormal{Supp}(N)$,
the natural map
$$
\mathcal O_Y \ra g_* \mathcal O(E),
$$
is an isomorphism.
\end{enumerate}
Then the minimal extremal face of $\Lambda_{\eff}(X)$ containing $D$ is 
generated by vertical divisors of $g$ and components of $N$.
\end{prop}

\begin{proof}
When $D$ is big, the assertion is trivial.
We may assume that $\dim(Y) < \dim(X)$.
Let $F$ be the minimal extremal face of $\Lambda_{\eff}(X)$ containing $D$.
Since $F$ is extremal, it follows that $F$ contains all vertical divisors of $g$
and components of $N$.

On the other hand, our assumption implies
that for general fiber $X_y$, $N|_{X_y}$ is a rigid divisor on $X_y$,
and its irreducible components generate an extremal face $F'$ of $\Lambda_{\eff}(X_y)$.
Let $\alpha \in \overline{\mathrm{NM}}_1(X_y)$ be a nef cycle supporting $F'$
and $F_\alpha := \{ \alpha = 0 \} \cap \Lambda_{\eff}(X)$.
Since $X_y$ is a general fiber, $\alpha \in \overline{\mathrm{NM}}_1(X)$
so that $F_\alpha$ is an extremal face. Since
$D\cdot \alpha =0$ and $F$ is 
minimal we have $F \subset F_\alpha$.
Let $D' \in F$ be an effective $\bQ$-divisor. Since
$D' \cdot \alpha = 0$, $D'$ is a sum of vertical divisors of $g$
and components of $N$;  our assertion follows.
\end{proof}

\begin{prop}
\label{prop:weakbalance_toric}
Let $X$ be a projective toric variety and $Y$ an 
equivariant compactification of a subtorus of codimension one (possibly singular).
Let $L$ be a big line bundle on $X$.
Then $L$ is weakly balanced with respect to $Y$.
\end{prop}

\begin{proof}
Let $M$ be the class of $\mathcal O_X(Y)$. 
Applying an equivariant embedded resolution of singularities, if necessary,
we may assume that $X$ and $Y$ are smooth or at least $\bQ$-factorial terminal.
Due to a group action of a torus, $Y$ is not rigid, so that
$$
a(X, L)L|_Y + K_Y = (a(X, L)L + K_X)|_Y + M|_Y \in \Lambda_{\eff}(Y).
$$
Note that $a(X, L)L + K_X$ is an effective $\bQ$-divisor on $X$.
Thus we have
$$
a(Y, L|_Y) \leq a(X, L).
$$
Suppose that $a(Y, L|_Y) = a(X, L) =: a$.
Let $D = aL + K_X + Y$
and consider the Zariski decomposition of $D$:
\

\centerline{
\xymatrix{
& &   \tilde{X}\ar[dl]_{\mu} \ar[d]_{\nu} & \!\!\!\!\!\!\!\!\!\!\!\!\!\!\!\!\!\!\!\! \supset  \tilde{Y}\\
D \subset \!\!\!\!\!\!\!\!\!\!\!\!\!\!\!\!\!\!& X\ar@{-->}[r]_f  & X' &  \!\!\!\!\!\!\!\!\!\!\!\!\!\!\!\! \supset D',   
}
}

\

\noindent
where $D'$ is the strict transform of $D$, which is nef,
and $\tilde{Y}$ is the strict transform of $Y$.
We may assume that both $\tilde{X}$ and $\tilde{Y}$ are smooth.
Let $F$ be the minimal extremal face of $\Lambda_{\eff}(\tilde{X})$ containing
$$
a\mu^*L+K_{\tilde{X}} +\tilde{Y}.
$$
Since $a\mu^*L+K_{\tilde{X}} \in F$, it follows that
$
\mathrm{codim} (F) \leq b(X, L).
$
Since $X$ has only terminal singularities, we have
$$
a\mu^*L+K_{\tilde{X}} +\tilde{Y} = a\mu^*L+\mu^*K_X + \sum_i d_iE_i +\tilde{Y},
$$ 
where $d_i$'s are positive integers
and $E_i$'s are $\mu$-exceptional divisors.
It follows that $F$ is the minimal extremal face containing $\mu^*D$
and all $\mu$-exceptional divisors.
Let $g : \tilde{X} \ra B$ be the semi-ample fibration associated to $\nu^*D'$.
Note that $\dim(B) < \dim (\tilde{X})$
since $D$ is not big.
Proposition \ref{prop:beta_MDS} implies that
$F$ is generated by all vertical divisors of $g$ and all $\nu$-exceptional divisors.
We denote the vector space, generated by $F$, by $V_F$.

Let $F'$ be the minimal extremal face of $\Lambda_{\eff}(\tilde{Y})$
containing 
$$
a\mu^*L|_{\tilde{Y}} + K_{\tilde{Y}} = (a\mu^*L + K_{\tilde{X}} + \tilde{M})|_{\tilde{Y}},
$$
where $\tilde{M}$ be the class of $\mathcal O_{\tilde{X}}(\tilde{Y})$.
Then $F'$ is also the minimal extremal face containing $\mu^*D|_{\tilde{Y}}$
and all components of $(E_i \cap \tilde{Y})$'s
so that $F'$ is the minimal extremal face containing $\nu^*D'|_{\tilde{Y}}$
and all components of $(G_j \cap \tilde{Y})$'s,
where $G_j$'s are all $\nu$-exceptional divisors.
In particular, $F'$ contains all vertical divisors of $g|_{\tilde{Y}} : \tilde{Y} \ra H = g(\tilde{Y})$.
Since $\nu^*D'$ admits a section vanishing along $\tilde{Y}$,
$H$ is a Weil divisor of $B$, which is a subtoric variety.
Let $V' \subset \NS(\tilde{Y})$ be a vector space generated by 
vertical divisors of $g|_{\tilde{Y}}$ and components of $(G_j \cap \tilde{Y})$'s.
Then $b(Y, L) \leq \mathrm{codim}(V')$.
Consider the following restriction map:
$$
\Phi : \NS(\tilde{X}) / V_F \ra \NS(\tilde{Y})/ V'.
$$
We claim that $\Phi$ is surjective.
Let $N$ be an irreducible component of the boundary divisor of $\tilde{Y}$
which dominates $H$.
There exists an irreducible component $N'$ of the boundary divisor of $\tilde{X}$
such that $N'$ contains $N$.
Then $N'$ also dominates $B$.
As in the proof of Lemma~\ref{lemm:domin},
$$
N' \cap \tilde{Y} = mN + (\text{vertical divisors of $g|_{\tilde{Y}}$}).
$$
Our claim follows from this.
Hence
$$
b(Y, L) \leq \mathrm{codim}(V') \leq \mathrm{codim}(V_F) \leq b(X, L).
$$
\end{proof}

\begin{prop}
\label{prop:notbalance_toric}
Let $X$ be a $\bQ$-factorial terminal projective toric variety
and $L$ a big line bundle on $X$.
Suppose that the positive part of Zariski decomposition of $D:= a(X, L)L+K_X$ is nontrivial.
Then $L$ is not balanced.
\end{prop}

\begin{proof}
After blowing up, if necessary,
we may assume that $D$ itself admits a Zariski decomposition
$$
D = P + N,
$$
where $P$ is a nef $\bQ$-divisor and $N \geq 0$ is the negative part.
Let $g : X \ra Y$ be the semi-ample fibration associated to $P$.
We consider a general fiber $X_y$ of $g$.
Since $a(X, L)L|_{X_y} + K_{X_y} = N|_{X_y}$ is a rigid effective divisor,
we conclude that $a(X, L) = a(X_y, L|_{X_y})$.
Let $V \subset \NS(X)$ be the vector space generated by
vertical divisors of $g$ and components of $N$
and $V' \subset \NS(X_y)$ the vector space generated by
components of $N|_{X_y}$.
The restriction map
$$
\Phi : \NS(X) / V \ra \NS (X_y) / V',
$$
is surjective, by
Lemma~\ref{lemm:domin}.
On the other hand, let $T$ be the big torus of $Y$.
Then the preimage $g^{-1}(T)$ of $T$ is a product of $T$
and a general fiber $X_y$.
It follows that $\Phi$ is injective.
Thus we have
$$
b(X, L) = b(X_y, L|_{X_y}).
$$
Hence $L$ is not balanced on $X$.
\end{proof}

An alternative proof of Theorem \ref{thm:subgroup} for toric varieties is provided below:
\begin{prop}
\label{prop:rigid_toric}
Let $X$ be a $\bQ$-factorial terminal projective toric variety,
$L$ a big line bundle on $X$,
and $Y$ an equivariant compactification of a subtorus of codimension one (possibly singular).
Suppose that the positive part of the Zariski decomposition of $a(X, L)L + K_X$ is trivial.
Then $L$ is balanced with respect to $Y$.
\end{prop}
\begin{proof}
We follow the notations in the proof of Proposition~\ref{prop:weakbalance_toric}.
We only need to explain why $b(Y, L|_Y) < b(X, L)$,
when $a(Y, L|_Y) = a(X, L)$.
Since $a\mu^*L + K_{\tilde{X}}$ is rigid,
it follows that 
$$
\mathrm{codim}(V_F) < b(X, L).
$$
Thus our assertion follows.
\end{proof}

\begin{coro}
\label{coro:canonical_toric}
Let $X$ be a $\bQ$-factorial terminal projective toric variety.
A big line bundle $L$ is balanced with respect to all toric subvarieties
if and only if $a(X, L)L + K_X$ is rigid.
\end{coro}

\begin{rema}
\label{rema:general}
Propositions \ref{prop:weakbalance_toric} and \ref{prop:rigid_toric}
hold when $X$ is $\bQ$-factorial terminal
and $Y$ a general {\em smooth} divisor.
The proofs work with small modifications.
However, we still do not know whether they hold for {\em singular} divisors.
\end{rema}

\begin{exam}
\label{exam:no_control}
Consider the standard action of $\mathbb G_m^3 = \{ (t_0, t_1, t_2) \}$ on $\bP^3$ by
$$
(t_0, t_1, t_2)\cdot (x_0:x_1:x_2:x_3) \mapsto (t_0x_0:t_1x_1:t_2x_2:x_3).
$$
Consider the subtorus
$$
M = \{ (t_0, t_1, (t_0t_1)^{-1}) \} \subset \mathbb G_m^3,
$$
and let $S$ be the equivariant compactification of $M$ defined by
$$
x_0x_1x_2 = x_3^3.
$$
This is a singular cubic surface with three isolated singularities of type $\mathsf A_2$.  
We denote them by $p_1, p_2, p_3\in \mathbb P^3$. Since they are fixed under the action of 
$\mathbb G_m^3$ on $\bP^3$, the blowup $B := \Bl_{p_1, p_2, p_3}(\bP^3)$ 
is an equivariant compactification of $\mathbb G_m^3$. Moreover, the closure $\tilde{S}$ 
of $M$ in $B$ is the minimal desingularization of $S$ and 
the class of $\tilde{S}$ in $\Pic(B)$ is ample. 
Put $X:=B\times \bP^1$ and $Y:=\tilde{S}\times \bP^1$. 
We have a diagram

\

\centerline{
\xymatrix{ 
X\ar[d]_{\pi} &   \ar@{_{(}->}[l] Y\ar[d] \\
B            &    \ar@{_{(}->}[l] \tilde{S}, 
}
}

\

\noindent
Then $Y$ is a nef divisor, and we have
$$
\rk \, \NS(Y) = 8 > \rk \, \NS(X) = 5.
$$
However, the anticanonical class $-K_X$ is still balanced
with respect to $Y$ since
\begin{align*}
&a(Y, -K_X|_Y) = a(X, -K_X) =1\\
&b(Y, -K_X|_Y) = 1 < b(X, -K_X) = 5.
\end{align*}
This shows that, in general, we cannot expect to control the subgroup of $\NS(X)$ 
generated by vertical divisors. In the proof of Proposition~\ref{prop:weakbalance_toric}, we
were able to control the {\em quotient} by this subgroup.   
\end{exam}

\bibliographystyle{alpha}
\bibliography{balanced}

\end{document}